\documentclass{amsart}
\usepackage{amssymb, amsmath, amscd, latexsym, mathrsfs, appendix}
\usepackage{graphics}
\usepackage{amsmath}
\usepackage{amsxtra}
\usepackage{amsfonts}
\usepackage{appendix}
\usepackage[all]{xy}

\usepackage{color}

\usepackage{tikz}
\usetikzlibrary{matrix,arrows}

\usepackage{hyperref}
\hypersetup{
    colorlinks,
    citecolor=black,
    filecolor=black,
    linkcolor=black,
    urlcolor=black
}

\newtheorem{thm}{Theorem}[]

\newtheorem{cor}[thm]{Corollary}

\newcommand{\map}{\textup{map}}

\newcommand{\emb}{\textup{emb}}
\newcommand{\config}{\mathsf{con}}

\newcommand{\loc}{\textup{loc}}
\newcommand{\imm}{\textup{imm}}

\newcommand{\fin}{\mathsf{Fin}}
\newcommand{\finplus}{{\fin_*}}

\newcommand{\RR}{\mathbb R}

\newcommand{\inj}{\textup{injmap}}

\begin{document}

\title{Topological embeddings of disks via configuration categories}
\author{Pedro Boavida de Brito and Michael S. Weiss}%

\address{Dept. of Mathematics, Instituto Superior Tecnico, Univ. of Lisbon, Av. Rovisco Pais, Lisboa, Portugal}%
\email{pedrobbrito@tecnico.ulisboa.pt}

\address{Mathematics Institute, WWU M\"{u}nster, 48149 M\"{u}nster, Einsteinstrasse 62, Germany}%
 \email{m.weiss@uni-muenster.de}

\thanks{M.W. gratefully acknowledges the support of the Bundesministerium f\"ur Bildung und Forschung through the A.v.Humboldt foundation (Humboldt professorship, 2012-2017). P.B. gratefully acknowledges the support of FCT through grant SFRH/BPD/99841/2014.}

\begin{abstract} 
The functor that takes a manifold to its configuration category exhibits a type of full faithfulness in some cases.
\end{abstract}
 
 \maketitle
 
 Let $N^n$ be a smooth manifold with boundary. Fix a smooth embedding $b$ from $\partial D^m$ to $\partial N$. We write $\inj_\partial(D^m, N)$ for the space of injective maps whose restriction to the boundary is $b$ and $\inj^s_\partial(D^m, N)$ for the union of path components which contain smooth embeddings. Similarly, we write $\emb^{TOP}_{\partial}(D^m, N)$ for the space of locally flat topological embeddings and $\emb^{TOP,s}_{\partial}(D^m, N)$ for the union of path components which contain smooth embeddings. We assume that $n - m > 3$ and $n \geq 5$. Under these circumstances, the inclusion of $\emb^{TOP}_{\partial}(D^m, N)$ in $\inj_\partial(D^m, N)$ is a weak equivalence (see \cite{Lashof}).

\medskip 

The following result is essentially an extension of \cite[Theorem 10.1]{BoavidaWeiss2}. We keep the notation from that paper.

\begin{thm} The evaluation map
\[
\inj^s_\partial(D^m, N) \to \RR \map_{\finplus}^\partial(\config(D^m), \config(N))
\]
is an almost weak equivalence, i.e. all its homotopy fibers are either empty or contractible.
\end{thm}
 
\begin{proof}
Let $e$ be a neat smooth embedding $D^m \hookrightarrow N$ of an $m$-disk in $N$ which extends $b$. We extend $e$ to a smooth embedding $f : D^m \times D^{n-m} \to N$ by taking a normal tube so that $f^{-1}(\partial N)$ is $S^{m-1} \times D^{n-m}$. Composition with $f$ gives a map of squares from
\begin{equation} \label{eq:square1}
	\begin{tikzpicture}[descr/.style={fill=white}, baseline=(current bounding box.base)] ]
	\matrix(m)[matrix of math nodes, row sep=2.5em, column sep=2.5em,
	text height=1.5ex, text depth=0.25ex]
	{
	\emb^{TOP,s}_{\partial}(D^m,D^n) & \RR \map_{\finplus}^\partial(\config(D^m), \config(D^n)) \\
	\imm^{TOP,s}_{\partial}(D^m,D^n) & \RR \map_{\finplus}^\partial(\config^\loc(D^m), \config^\loc(D^n)) \\
	};
	\path[->,font=\scriptsize]
		(m-1-1) edge node [auto] {} (m-1-2);
	\path[->,font=\scriptsize]
		(m-2-1) edge node [auto] {} (m-2-2);
	\path[->,font=\scriptsize]
		(m-1-1) edge node [left] {} (m-2-1);
	\path[->,font=\scriptsize] 		
		(m-1-2) edge node [auto] {} (m-2-2);
	\end{tikzpicture}
\end{equation}
to
\begin{equation} \label{eq:square2}
	\begin{tikzpicture}[descr/.style={fill=white}, baseline=(current bounding box.base)] ]
	\matrix(m)[matrix of math nodes, row sep=2.5em, column sep=2.5em,
	text height=1.5ex, text depth=0.25ex]
	{
	\emb^{TOP,s}_{\partial}(D^m,N) & \RR \map_{\finplus}^\partial(\config(D^m), \config(N)) \\
	\imm^{TOP,s}_{\partial}(D^m,N) & \RR \map_{\finplus}^\partial(\config^\loc(D^m), \config^\loc(N)) \\
	};
	\path[->,font=\scriptsize]
		(m-1-1) edge node [auto] {} (m-1-2);
	\path[->,font=\scriptsize]
		(m-2-1) edge node [auto] {} (m-2-2);
	\path[->,font=\scriptsize]
		(m-1-1) edge node [left] {} (m-2-1);
	\path[->,font=\scriptsize] 		
		(m-1-2) edge node [auto] {} (m-2-2);
	\end{tikzpicture}
\end{equation}
Both squares are cartesian, so we can view the map of squares as a cartesian cube. Then the lower face (consisting of the immersion spaces and spaces of maps of local configuration categories) is cartesian. So the top face is also cartesian. Furthermore, the top row of (\ref{eq:square1}) is a weak equivalence since both terms are contractible by \cite{BoavidaWeiss2}. Therefore, the homotopy fiber of the top row of (\ref{eq:square2}) over the basepoint (determined by $e$) is weakly contractible. Repeating the argument for different choices of $e$ we get the claim.
\end{proof}

\begin{cor} The commutative square
\begin{equation*}
	\begin{tikzpicture}[descr/.style={fill=white}, baseline=(current bounding box.base)] ]
	\matrix(m)[matrix of math nodes, row sep=2.5em, column sep=2.5em,
	text height=1.5ex, text depth=0.25ex]
	{
	\emb_{\partial}(D^m,N) & \inj^s_\partial(D^m, N) \\
	\imm_{\partial}(D^m,N) & \Omega^m \Gamma \\
	};
	\path[->,font=\scriptsize]
		(m-1-1) edge node [auto] {} (m-1-2);
	\path[->,font=\scriptsize]
		(m-2-1) edge node [auto] {} (m-2-2);
	\path[->,font=\scriptsize]
		(m-1-1) edge node [left] {} (m-2-1);
	\path[->,font=\scriptsize] 		
		(m-1-2) edge node [auto] {} (m-2-2);
	\end{tikzpicture}
\end{equation*}
is homotopy cartesian. Here $\Gamma$ is the space of pairs $(x, \alpha)$ where $x$ is a point in the interior of $N$ and $\alpha$ is a derived operad map $E_m \to E_{T_xN}$, and $E_{T_xN}$ is the  operad of little disks in $T_xN$ (it is abstractly isomorphic to $E_n$.)
\end{cor}
\begin{proof}
We obtain this from \cite[Theorem 1.1]{BoavidaWeiss2} in the boundary version and the theorem above.
\end{proof}

 \end{document}